\documentclass[11pt,a4paper]{article}

\usepackage{ifpdf}

\makeatletter
\let\@font@warningori\@font@warning
\newcommand\shutup{\def\@font@warning##1{}}
\newcommand\youcanspeak{\let\@font@warning\@font@warningori}
\makeatother

\usepackage{amsmath,amsbsy,bm,latexsym,enumerate}

\shutup
\usepackage{fourier}
\youcanspeak
\usepackage[scaled=0.875]{helvet}

\usepackage[latin1]{inputenc}
\usepackage[german,dutch,english]{babel}

\usepackage{amsthm}

\usepackage{graphicx,epsfig,color}
\usepackage{tikz}

\usepackage{anysize}

\usepackage{subfigure}

\providecommand{\norm}[1]{\lVert#1\rVert}

\providecommand{\setZ}{\mathbb{Z}}

\providecommand{\setR}{\mathbb{R}}

\newcommand {\Z}	  {\mathbb{Z}}

\newcommand {\R}	  {\mathbb{R}}

\newcommand{\vol}{\mathrm{vol}}

\renewcommand{\epsilon}{\varepsilon}

\renewcommand{\leq}{\leqslant}
\renewcommand{\geq}{\geqslant}

\theoremstyle{plain}
\newtheorem{theorem}{Theorem}
\newtheorem{lemma}[theorem]{Lemma}

\theoremstyle{definition}

\title{On sub-determinants and the diameter of  polyhedra\footnote{An
    extended abstract of this paper was presented at the 28-th annual
    ACM symposium on Computational Geometry (SOCG 12)} }%  is polynomial in the largest
\author{
  Nicolas Bonifas\thanks{LIX, École Polytechnique, Palaiseau and IBM, Gentilly (France). \texttt{bonifas@lix.polytechnique.fr}}
  \and
  Marco Di Summa\thanks{Dipartimento di Matematica, Università di Padova (Italy). \texttt{disumma@math.unipd.it}}
  \and
  Friedrich Eisenbrand\thanks{Ecole Polytechnique Fédérale de Lausanne (Switzerland). \texttt{friedrich.eisenbrand@epfl.ch}}
  \and
  Nicolai Hähnle\thanks{University of Bonn (Germany). \texttt{haehnle@or.uni-bonn.de }}
  \and
  Martin Niemeier\thanks{Technische Universität Berlin (Germany). \texttt{martin.niemeier@tu-berlin.de}}
}
\date{\today }

\begin{document}
\maketitle

\begin{abstract}
  \noindent
  We derive a new upper bound on the diameter of a   polyhedron
  $P = \{x \in \setR^n \colon Ax\leq b\}$, where $A \in \setZ^{m×n}$.
  The bound is polynomial in $n$ and the largest
  absolute value of a sub-determinant of $A$, denoted by $\Delta$. More precisely, we show
  that the diameter of $P$ is bounded by $O\left(\Delta^2 n^4\log
    n\Delta\right)$.  If $P$ is bounded, then we show that the diameter of
  $P$ is at most
  $O\left(\Delta^2 n^{3.5}\log n\Delta\right)$.

  For the special case in which $A$ is a totally
  unimodular matrix, the bounds are  $O\left(n^4\log
    n\right)$ and $O\left(n^{3.5}\log n\right)$ respectively. This
  improves  over the previous best bound of
  $O(m^{16}n^3(\log mn)^3)$ due to Dyer and
  Frieze~\cite{MR1274170}. %In particular, our bound does not depend on
%  the number of rows of $A$.
\end{abstract}

%\category{G.2}{Mathematics of Computing}{Discrete Mathematics}
%\keywords{Diameter of polyhedra, polyhedral graph, totally unimodular matrices, isoperimetric inequality}

\section{Introduction}
\label{sec:bound}

One of the fundamental  open problems in
optimization and discrete geometry is the question whether the diameter of
a polyhedron  can be bounded by a
polynomial in the dimension and the number of its defining inequalities.
% polyhedra. The diameter of a polyhedron is defined as the diameter of
% its polyhedral graph. Although being subject of intensive study for
% more than half a century, the problem remains far from being well
% understood.  A well known conjecture by Hirsch from
%1957 %TODO: References!
%
The  problem is readily explained:
A \emph{polyhedron} is a set of the form $P = \{x \in \setR^n \colon Ax\leq b\}$,
where $A \in \setR^{m×n}$ is a matrix and $b \in \setR^m$ is an
$m$-dimensional vector.
A \emph{vertex} of $P$ is a point $x^* \in P$ such that there
exist $n$ linearly independent rows of $A$ whose corresponding
inequalities of $Ax\leq b$ are satisfied by $x^*$ with equality.
Throughout this paper, we assume that the polyhedron $P$ is \emph{pointed}, i.e. it has vertices,
which is equivalent to saying that the matrix $A$ has full column-rank.
Two different
vertices $x^*$ and $y^*$ are \emph{neighbors} if they are the endpoints of an \emph{edge} of the polyhedron,
i.e. there exist $n-1$ linearly independent rows of $A$ whose corresponding inequalities of
$Ax\leq b$ are satisfied with equality both by $x^*$ and $y^*$.
In this
way, we obtain the undirected \emph{polyhedral graph} with edges being pairs of neighboring
vertices of $P$. This graph is connected. The \emph{diameter} of $P$ is the
smallest natural number that bounds the  length of a shortest path
between any pair of vertices in this graph.
The question is now as follows:
\begin{quote}
  Can the diameter of a polyhedron $P = \{x \in \setR^n \colon Ax\leq b\}$ be
  bounded by a polynomial in $m$ and $n$?
\end{quote}
The belief in a positive answer to this question is called
the \emph{polynomial Hirsch conjecture}.
Despite a lot of research effort during the last 50 years, the gap
between lower and upper bounds on the diameter remains huge.
While, when the dimension $n$ is fixed, the diameter can be bounded
by a linear function of $m$ \cite{MR0254735,MR0355826}, for the general case the best upper bound,
due to Kalai and Kleitman~\cite{MR1130448}, is $m^{1 + \log n}$. The best lower bound is of the form $(1+\varepsilon) \cdot m$ for
some $\varepsilon>0$ in fixed and sufficiently large dimension $n$. This is due
to a celebrated result of Santos~\cite{Santos10} who disproved the,
until then longstanding, original \textit{Hirsch conjecture} for polytopes. The Hirsch
conjecture stated that the diameter of a bounded polyhedron\footnote{A
counterexample to the same conjecture for unbounded polyhedra
was found in 1967 by Klee and Walkup~\cite{KleeWalkup}.}
is at most $m-n$.
Interestingly, this huge gap (polynomial versus
quasi-polynomial)  is also not closed in a very simple
combinatorial abstraction of polyhedral graphs~\cite{EHRR10}.
However, it was shown by Vershynin~\cite{MR2529774} that
every polyhedron can be perturbed by a small random amount so that
the expected diameter of the resulting polyhedron is bounded by
a polynomial in  $n$ and $\ln m$.
See Kim and Santos~\cite{MR2681516} for a recent survey.

In light of the importance and apparent difficulty  of the open
question above, many researchers have shown that it can be answered in
an affirmative way in some special cases. Naddef~\cite{Naddef89}
proved that the Hirsch conjecture holds true for
$0/1$-polytopes. Orlin~\cite{MR1486304} provided a
quadratic upper bound for flow-polytopes. Brightwell et
al.~\cite{MR2223631} showed that the diameter of the
transportation polytope is linear in $m$ and $n$,
and a similar result holds for the dual of a transportation
polytope~\cite{MR769400} and the axial $3$-way transportation polytope~\cite{MR2568801}.

The results on flow polytopes and classical transportation polytopes concern
polyhedra defined by \textit{totally unimodular matrices},
i.e., integer matrices whose sub-determinants are $0,± 1$. For
such polyhedra Dyer and Frieze~\cite{MR1274170} had previously shown
that the diameter is bounded by a polynomial in $n$ and $m$. Their
bound is $O(m^{16}n^3(\log mn)^3)$. Their result is also
algorithmic: they show that there exists a randomized
simplex-algorithm that solves  linear programs defined by totally unimodular matrices
in polynomial time.

Our main result is a generalization and considerable improvement of the diameter
bound of Dyer and Frieze.  We show that the diameter of a polyhedron
$P = \{x \in \setR^n \colon Ax\leq b\}$, with $A \in \setZ^{m×n}$ is bounded by
$O\left(\Delta^2 n^4\log n\Delta\right)$. Here, $\Delta$ denotes the largest
absolute value of a \textit{sub-determinant} of $A$.  If $P$ is bounded,
i.e., a \textit{polytope},
then we can show that the diameter of $P$ is at most $O\left(\Delta^2
  n^{3.5}\log n\Delta\right)$.  To compare our bound with the one of Dyer
and Frieze one has to set $\Delta$ above to one and obtains
$O\left(n^4\log n\right)$ and $O\left(n^{3.5}\log n\right)$
respectively.  Notice that our bound
is  independent of $m$, i.e., the number of rows of $A$.

\subsection*{The proof method}

Let $u$ and $v$ be two vertices of $P$.
We  estimate the maximum number of  iterations of two   breadth-first-search
explorations of the polyhedral graph, one initiated at $u$, the other
initiated at $v$, until a common vertex is discovered. The diameter of
$P$ is at most twice this number of iterations.
The main idea in the analysis is to reason about the normal cones of vertices of $P$
and to exploit a certain volume expansion property.

% In the following
% we describe this in greater detail.
%Expansion-properties are also crucial in
% random-walk techniques and in particular also in the algorithm of Dyer
% and Frieze~\cite{MR1274170}. The graphs of our polyhedra however do not
% have such an expansion property. We are not
% interested in an algorithmic result however and can argue in a
% geometrical direct way.   \marginpar{@Nicolai: Kannst u hier noch was schreiben?} In
% fact, our technique is simple and requires only very basic geometry
% in the case, where $P$ is bounded. In the case, where $P$ is
% unbounded, an isoperimetric inequality Lovász and
% Simonovits~\cite{lovasz1993random} comes will come at help. We now
% summarize the technique.

We can assume that $P = \{x \in \setR^n \colon Ax\leq b\}$ is
\textit{non-degenerate}, i.e., each vertex has exactly $n$ tight
inequalities. This can be achieved by slightly perturbing the right-hand side vector $b$: in
this way the diameter can only grow. Notice that the polyhedron is then also full-dimensional.
We denote the polyhedral graph of $P$ by $G_P = (V,E)$.
Let $v \in V$ now be a vertex of $P$.
The \textit{normal cone} $C_v$ of $v$ is the set of all vectors $c \in
\setR^n$ such that $v$ is an optimal solution of the linear program
$\max\{c^Tx\colon x \in \setR^n, \, Ax\leq b\}$. % The normal cone is a
% \textit{cone}, i.e., a set  that is closed under multiplication
% with a non-negative   scalar.
The normal cone $C_v$ of a vertex of $v$ is a full-dimensional simplicial
polyhedral cone.  Two vertices $v$ and $v'$ are adjacent if and only
if $C_v$ and $C_{v'}$ share a facet. No two distinct normal cones share an interior point.
Furthermore, if $P$ is a polytope, then the union of the
normal cones of vertices of $P$ is the complete space $\setR^n$.

We now define the \textit{volume} %  the usual notion of volume in $\R^n$ by defining the
% volume
of a set $U\subseteq V$ of vertices as
the volume of the union of the normal cones of $U$ intersected with the
\textit{unit ball} $B_n = \{ x \in \setR^n \colon \|x\|_2\leq1\}$, i.e.,
\begin{displaymath}
  \vol(U) := \vol \left( \bigcup_{v \in U} C_v \cap B_n \right).
\end{displaymath}
Consider an iteration of breadth-first-search. Let $I\subseteq V$ be the set of
vertices that have been discovered so far. Breadth-first-search will next
discover the neighborhood of $I$, which we denote by $\mathcal{N}(I)$.

Together with the integrality of $A$, the bound $\Delta$ on the subdeterminants
guarantees that the angle between one facet of a normal cone $C_v$ and the opposite ray is not too small.
We combine this fact, which we formalize in Lemma~\ref{lemma:UpperBoundOnDockingSurface},
with an isoperimetric inequality to show that the volume of $\mathcal{N}(I)$ is large relative to the volume of $I$.

\begin{lemma}
\label{lemma:VolumeOfNeighborhood}
Let $P=\{x \in \setR^n \colon Ax\leq b\}$ be a polytope where all
sub-determinants of $A\in \setZ^{m×n}$ are bounded by $\Delta$ in absolute value
and let  $I\subseteq V$ be a set of vertices of
$G_P$ with  $\vol(I) \leq (1/2) \cdot \vol(B_n)$.
Then the volume of the neighborhood of $I$ is at least
$$\vol({\mathcal{N}(I)})\geq \sqrt{\frac{2}{\pi}}\frac{1}{\Delta^2 n^{2.5}}\cdot \vol(I).$$
\end{lemma}
We provide the proof of this lemma in the next section.
Our diameter bound for polytopes is an easy consequence:
\begin{theorem}
\label{thr:1}
Let $P=\{x \in \setR^n \colon Ax\leq b\}$ be a polytope where all
subdeterminants of $A\in \setZ^{m×n}$ are bounded by $\Delta$ in absolute value.
The diameter of $P$ is bounded by $O\left(\Delta^2 n^{3.5}\log n\Delta\right)$.
%In particular, if $A$ is totally unimodular, then the diameter of $P$ is bounded by $O(n^{3.5}\log n)$.
\end{theorem}
\begin{proof}
  We estimate the maximum number of iterations of
  breadth-first-search until the total volume of the discovered
  vertices exceeds $(1/2) \cdot \vol(B_n)$. This is an upper bound on the
  aforementioned maximum number of iterations of two
  breadth-first-search explorations until a common vertex is
  discovered.

  Suppose we start at vertex $v$ and
  let $I_j$
  be the vertices that have been discovered during the first $j$
  iterations. We have   $I_o = \{v\}$. If $j\geq1$ and  $\vol(I_{j-1}) \leq (1/2) \cdot
  \vol(B_n)$ we have by Lemma~\ref{lemma:VolumeOfNeighborhood}
  \begin{eqnarray*}
    \vol(I_{j}) & \geq &  \left(1+ \sqrt{\frac{2}{\pi}}\frac{1}{\Delta^2 n^{2.5}}\right)
    \vol(I_{j-1}) \\
              & \geq & \left(1+ \sqrt{\frac{2}{\pi}}\frac{1}{\Delta^2 n^{2.5}}\right)^{j}
              \vol(I_0).
  \end{eqnarray*}
The condition $\vol(I_j) \leq (1/2) \cdot \vol(B_n)$ implies
  \begin{equation*}
%    \label{eq:1}
    \left(1+ \frac{1}{ \sqrt{\frac{\pi}{2}} \Delta^2 n^{2.5}}\right)^{j}
    \vol(I_0) \leq 2^n.
  \end{equation*}
  This is equivalent to
  \begin{equation*}
%    \label{eq:2}
    j \cdot \ln \left(1+ \frac{1}{ \sqrt{\frac{\pi}{2}} \Delta^2 n^{2.5}}\right)
    \leq \ln (2^n / \vol(I_0)).
  \end{equation*}
For $0 \leq x \leq 1$ one has $\ln (1 + x)\geq x/2$ and thus the inequality
above implies
\begin{equation}
\label{eq:3}
  j \leq   \sqrt{2 \pi} \Delta^2 n^{2.5}\cdot  \ln (2^n / \vol(I_0)).
\end{equation}
To finish the proof we need a lower bound on $\vol(I_0)$, i.e., the
$n$-dimensional volume of the set $C_v\cap B_n$. The normal cone $C_v$
contains the full-dimensional simplex spanned by $0$ and the
$n$ row-vectors $a_{i_1},\ldots,a_{i_n}$
 of $A$   that correspond to the inequalities of $Ax\leq b$ that are tight at
$v$. Since $A$ is integral, the volume of this simplex is at least
$1/n!$. Furthermore, if this simplex is scaled by $1 / \max\{
\|a_{i_k}\| \colon k=1,\ldots,n\}$,  then it is contained in the unit
ball. Since each component of $A$ is itself a sub-determinant, one has
$\max\{\|a_{i_k}\| \colon k=1,\ldots,n\} \leq \sqrt{n} \Delta$ and thus  $\vol(I_0)
\geq 1/ (n! \cdot n^{n/2} \Delta^n)$.
It follows that \eqref{eq:3} implies  $j = O\left(\Delta^2 n^{3.5}\log n\Delta\right)$.
\end{proof}

\textit{Remarks.}
The result of Dyer and Frieze~\cite{MR1274170} is also based on
analyzing expansion properties via isoperimetric inequalities. It is
our choice of normal cones as the natural geometric representation,
and the fact that we only ask for volume expansion instead of expansion of
the graph itself, that allows us to get a better bound.
Expansion properties of the graphs of general classes of
polytopes have also been studied elsewhere in the literature, e.g.~\cite{MR1116365, MR2077562}.

\subsection*{Organization of the paper}

The next section is devoted to a proof of the volume-ex­pan­sion
property, i.e., Lemma~\ref{lemma:VolumeOfNeighborhood}. The main tool
that is used here is a classical \textit{isoperimetric inequality} that
states that among measurable subsets of a sphere with fixed volume,
spherical caps have the smallest circumference.
Section~\ref{sec:unbounded} deals
with unbounded polyhedra. Compared to the case of polytopes, the
problem that arises here is the fact that the union of the normal cones is
not the complete space $\setR^n$. To tackle this case, we rely on an
isoperimetric inequality of Lovász and
Simonovits~\cite{lovasz1993random}. Finally, we discuss how our bound
can be further generalized. In fact, not all sub-determinants of $A$ need to
be at most $\Delta$ but merely the entries of $A$ and the
$(n-1)$-dimensional sub-determinants have to be bounded by $\Delta$, which
yields a slightly stronger result.

\section{Volume expansion}
\label{sec:VolumeExpansion}

This section is devoted to a proof of
Lemma~\ref{lemma:VolumeOfNeighborhood}. Throughout this section, we
assume that $P = \{ x \in \setR^n \colon Ax\leq b\}$ is a polytope. We begin with some useful
notation. A (not necessarily convex) \textit{cone} is a subset of
$\setR^n$ that is closed under the multiplication with non-negative
scalars. The intersection of a cone with the unit ball $B_n$ is called
a \textit{spherical cone}. Recall that $C_v$ denotes the normal cone of
the vertex $v$ of $P$.  We denote the spherical cone $C_v \cap B_n$ by
$S_v$ and, for a subset $U\subseteq V$, the spherical cone $\bigcup_{v \in U} S_v$ by
$S_U$. Our goal is to show that the following inequality holds for each
$I\subseteq V$ with $\vol(S_I) \leq \frac{1}{2} \vol(B_n)$:
\begin{equation}
  \label{eq:4}
  \vol(S_{\mathcal{N}(I)}) \geq \sqrt{\frac{2}{\pi}} \frac{1}{\Delta^2 n^{2.5}}
  \cdot \vol(S_I).
\end{equation}

Recall that two vertices are adjacent in $G_P$ if and only if their
normal cones have a common facet. This means that the neighbors of $I$ are those
vertices $u$ for which $S_u$ has a facet which is part of the surface
of    the
spherical cone $S_I$. In an iteration of breadth-first-search we thus
augment the set of discovered vertices $I$ by those vertices $u$ that
can ``dock'' on $S_I$ via a common facet. We call the $(n-1)$-dimensional
volume of the  surface  of a spherical cone $S$ that is  not on the sphere,
the \textit{dockable surface} $D(S)$, see Figure~\ref{fig:cone-regions}.

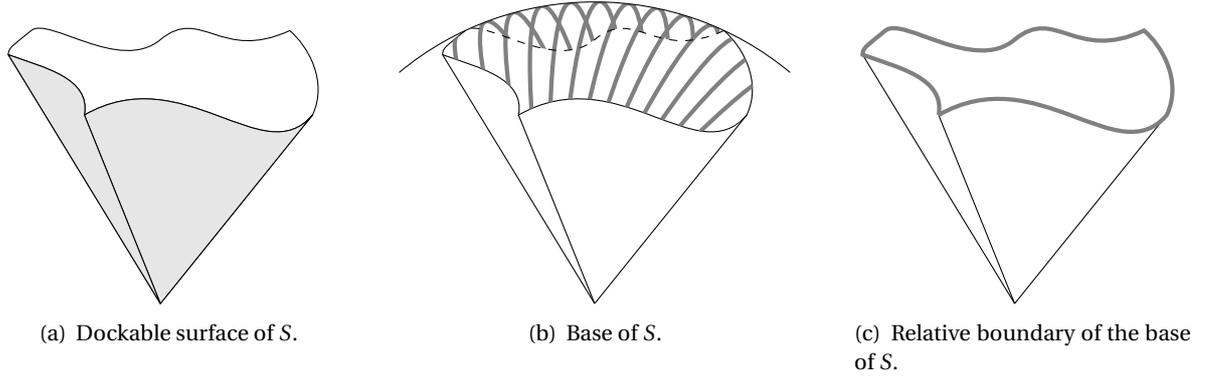
\begin{figure*}
\centering
  \subfigure[Dockable surface of $S$.]{
  \begin{tikzpicture}
    \coordinate (a) at (-1,2.5);
    \coordinate (b) at (2,2.5);
    \coordinate (c) at (1.7,3.62);
    \coordinate (d) at (0,3.5);
    \coordinate (e) at (-1.7,3.62);
    \coordinate (f) at (-2,3.3);
    \coordinate (o) at (0,0);
    \draw (o) -- (f);
    \draw (o) -- (a);
    \draw (o) -- (b);
    \draw (a) .. controls +(30:1.5) and +(230:1) .. (b) .. controls +(65:0.5) and +(320:0.3) .. (c) .. controls +(205:1) and +(40:0.7) .. (d) .. controls +(220:0.7) and +(40:0.3) .. (e) .. controls +(220:0.3) and +(90:0.1) .. (f) .. controls +(340:0.5) and +(80:0.5) .. (a);
    \draw[fill=gray!20,draw] (o) -- (a) .. controls +(30:1.5) and +(230:1) .. (b) -- cycle;
    \draw[fill=gray!20,draw] (o) -- (f) .. controls +(340:0.5) and +(80:0.5) .. (a) -- cycle;
  \end{tikzpicture}
  }
  \qquad
  \subfigure[Base of $S$.]{
  \begin{tikzpicture}
    \coordinate (a) at (-1,2.5);
    \coordinate (b) at (2,2.5);
    \coordinate (c) at (1.7,3.62);
    \coordinate (d) at (0,3.5);
    \coordinate (e) at (-1.7,3.62);
    \coordinate (f) at (-2,3.3);
    \coordinate (o) at (0,0);
    \draw (o) -- (f);
    \draw (o) -- (a);
    \draw (o) -- (b);
    \draw[draw] (e) .. controls +(220:0.3) and +(90:0.1) .. (f) .. controls +(340:0.5) and +(80:0.5) .. (a) .. controls +(30:1.5) and +(230:1) .. (b) .. controls +(65:0.5) and +(320:0.3) .. (c);
    \begin{scope}[yshift=114]
      \clip (c) .. controls +(205:1) and +(40:0.7) .. (d) .. controls +(220:0.7) and +(40:0.3) .. (e) arc (115.15:64.85:4);
      \foreach \k in {0,...,13} {
        \draw[ultra thick,gray,rotate around={\k*5.5-45:(0,-4)}] (0,0) arc (90:0:0.5 and 4);
      }
    \end{scope}
    \begin{scope}[yshift=114]
      \clip (e) .. controls +(220:0.3) and +(90:0.1) .. (f) .. controls +(340:0.5) and +(80:0.5) .. (a) .. controls +(30:1.5) and +(230:1) .. (b) .. controls +(65:0.5) and +(320:0.3) .. (c) arc (64.85:115.15:4);
      \foreach \k in {0,...,13} {
        %\draw[ultra thick,gray,rotate around={33-\k*5.5:(\k*5.5+57:4)+(0,-4)}] (\k*5.5+57:4) arc (90:180:0.5 and 4);
        \draw[ultra thick,gray,rotate around={\k*5.5-45:(0,-4)}] (0,0) arc (90:180:0.5 and 4);
      }
    \end{scope}
    \draw[dashed] (c) .. controls +(205:1) and +(40:0.7) .. (d) .. controls +(220:0.7) and +(40:0.3) .. (e);
    \draw (50:4) arc (50:130:4);
  \end{tikzpicture}
  }
  \qquad
  \subfigure[Relative boundary of the base of $S$.]{
  \begin{tikzpicture}
    \coordinate (a) at (-1,2.5);
    \coordinate (b) at (2,2.5);
    \coordinate (c) at (1.7,3.62);
    \coordinate (d) at (0,3.5);
    \coordinate (e) at (-1.7,3.62);
    \coordinate (f) at (-2,3.3);
    \coordinate (o) at (0,0);
    \draw (o) -- (f);
    \draw (o) -- (a);
    \draw (o) -- (b);
    \draw[gray,ultra thick] (a) .. controls +(30:1.5) and +(230:1) .. (b) .. controls +(65:0.5) and +(320:0.3) .. (c) .. controls +(205:1) and +(40:0.7) .. (d) .. controls +(220:0.7) and +(40:0.3) .. (e) .. controls +(220:0.3) and +(90:0.1) .. (f) .. controls +(340:0.5) and +(80:0.5) .. (a);
  \end{tikzpicture}
  }
  \caption{Illustration of $D(S)$, $B(S)$ and $L(S)$.}
  \label{fig:cone-regions}
 \end{figure*}
The  \textit{base} of $S$ is the  intersection of $S$ with
 the unit sphere. We denote the area  of the base by $B(S)$. By
 \textit{area} we mean the $(n-1)$-dimensional measure of some
 surface. Furthermore,
 $L(S)$ denotes the length of the relative boundary of the base
 of $S$. We use the term \textit{length} to denote the measure of an
 $(n-2)$-dimensional volume, see Figure~\ref{fig:cone-regions}.

Given any spherical cone $S$ in the unit ball, %with volume, area and
%length,
the following well-known relations follow from basic
integration:
\begin{equation}\label{eq:cones-formulas}
  \vol(S)=\frac{B(S)}n, \quad D(S)=\frac{L(S)}{n-1}.
\end{equation}

To obtain the volume expansion relation~\eqref{eq:4} we need to bound
the dockable surface of a spherical cone from below by its volume and,
for a simplicial spherical cone, we need an upper bound on the
dockable surface by its volume. More precisely, we show that
for every simplicial spherical cone $S_v$ one has
\begin{equation}
  \label{eq:5}
  \frac{D(S_v)}{\vol(S_v)} \leq \Delta^2 n^3
\end{equation}
and for any spherical cone %with volume and  dockable surface
one has
\begin{equation}
  \label{eq:6}
  \frac{D(S)}{\vol(S)} \geq \sqrt\frac{2n}{\pi}.
\end{equation}

Once inequalities~\eqref{eq:5} and \eqref{eq:6} are derived, the
bound~\eqref{eq:4} can be obtained as follows. All of the dockable
surface of $S_I$ must be ``consumed'' by the neighbors of
$I$. Using~\eqref{eq:6} one has thus
\begin{equation}
  \label{eq:7}
  \sum_{v \in \mathcal{N}(I)} D(S_v) \geq D(S_I) \geq \sqrt\frac{2n}{\pi}
  \cdot \vol(S_I).
\end{equation}
On the other hand, \eqref{eq:5} implies
\begin{equation}
  \label{eq:8}
  \sum_{v \in \mathcal{N}(I)}D(S_v) \leq  \Delta^2 n^3 \cdot \sum_{v \in
    \mathcal{N}(I)} \vol(S_v) = \Delta^2 n^3 \cdot \vol(S_{\mathcal{N}(I)}).
\end{equation}
These last two inequalities imply inequality \eqref{eq:4}. The
remainder of this section is devoted to proving~\eqref{eq:5} and
\eqref{eq:6}.

\subsection{Area to volume ratio of a spherical simplicial cone}
\label{subsec:area-volume-ratio}

We will first derive inequality~\eqref{eq:5}.
% spherical cones $C_v$ for a  (by \textit{area} we mean the
% $(n-1)$-dimensional measure of some surface).  To be more precise,
% since two spherical cones are adjacent when they share a facet, we are
% interested in the area $D(C_v)$ of the lateral surface of $C_v$, i.e.
% the total area of its facets. In other words, $D(C_v)$ is the area of
% that part of the boundary of $C_v$ which is disjoint from the unit
% sphere (we recall that the unit sphere is defined as the boundary of
% the unit ball). We call this surface the \textit{dockable surface} of
% $C_v$.

\begin{lemma}
  \label{lemma:UpperBoundOnDockingSurface}
  Let $v$ be a vertex of $P$. One has
  \[ \frac{D(S_v)}{\vol(S_v)} \leq \Delta^2 n^3. \]
\end{lemma}

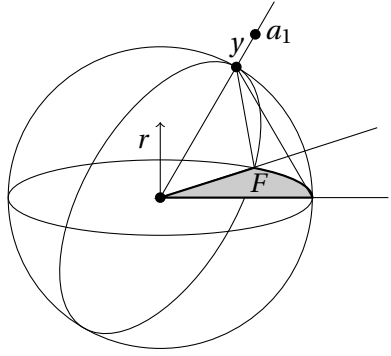
\begin{figure}
\centering
  \begin{tikzpicture}
    \coordinate (o) at (0,0);
    \coordinate (a) at (2,0);
    \coordinate (b) at (18:1.3);
    \coordinate [label=90:$y$] (y) at (60:2);
    \coordinate [label=0:$a_1$] (aj) at (60:2.5);
    \coordinate (r) at (0,1);

    \draw (0,0) circle (2cm);
    \draw[rotate=-30] (0,0) ellipse (1cm and 2cm);
    \draw (0,0) ellipse (2cm and 0.5cm);

    \draw (o) -- (60:3);
    \draw (o) -- (18:3);
    \draw (o) -- (0:3);

    \draw[thick, fill=gray!40] (o) -- (a) arc(0:52:2cm and 0.5cm) --cycle;
    \draw (1.3,0.2) node {$F$};
    \draw (a) -- (y) -- (b);

    \draw[->] (o) -- node[left,near end] {$r$} (r);

    \foreach \pt in {(o),(y),(aj)}
      \fill \pt circle (2pt);
  \end{tikzpicture}
  \caption{Proof of Lemma~\ref{lemma:UpperBoundOnDockingSurface}.}
  \label{fig:upper-bound-docking-surface}
\end{figure}

\begin{proof}
  Let $F$ be a facet of a spherical cone $S_v$. Let $y$ be the vertex
  of $S_v$ not contained in $F$.  Let $Q$ denote the convex hull of
  $F$ and $y$ (see Figure \ref{fig:upper-bound-docking-surface}).  We
  have $Q\subseteq S_v$ because $S_v$ is convex.  Moreover, if $h_F$ is the
  Euclidean distance of $y$ from the hyperplane containing $F$, then
  \[
    \vol(S_v) \geq \vol(Q) = \frac{\operatorname{area}(F)\cdot h_F}{n}.
  \]
  Summing over the facets of $S_v$, we find
  \begin{equation}
    \label{eq:docking-over-volume-first-step}
    \frac{D(S_v)}{\vol(S_v)} = \sum_{\text{facet } F} \frac{\operatorname{area}(F)}{\vol(S_v)} \leq n \cdot \sum_{\text{facet } F} \frac{1}{h_F}.
  \end{equation}
  It remains to provide a lower bound on $h_F$.
  Let $a_1, \ldots, a_n$ be the row-vectors of $A$ defining the extreme
  rays of the normal cone of $v$,  and let $A_v$ be the non-singular
  matrix whose rows are $a_1, \ldots, a_n$. Furthermore, suppose that
  the vertex $y$ lies on the ray generated by $a_1$.
  Let $H$ be the hyperplane generated by $a_2,\ldots,a_n$.
  The distance $d(y,H)$ of
  $y$ to $H$ is equal to $d(a_1,H) / \|a_1\|$.
   Let
  $b_1,\ldots,b_n$ be the columns of the adjugate of
  $A_v$. The column-vector $b_1$ is integral and each component of $b_1$ is
  bounded by $\Delta$. Furthermore $b_1$ is orthogonal to each of
  $a_2,\ldots,a_n$. Thus $d(a_1,H)$ is the length of the projection of
  $a_1$ to $b_1$, which is $|\langle a_1,b_1\rangle | / \|b_1\|\geq 1/(\sqrt{n} \cdot
  \Delta)$, since $a_1$ and $b_1$ are integral. Thus
  \begin{displaymath}
    h_F=d(y,H) \geq \frac{1}{n \Delta^2}.
  \end{displaymath}
  Plugging this into \eqref{eq:docking-over-volume-first-step} completes the proof.
\end{proof}

\subsection{An isoperimetric inequality for spherical cones}
\label{subsec:isoperimetric-union-cones}

We now derive the lower bound~\eqref{eq:6} on the area to volume ratio
for a general spherical cone. To do that, we assume
that the spherical cone has the least favorable shape for the area to
volume ratio and derive the inequality for cones of this shape.  Here
one uses classical isoperimetric inequalities. The basic isoperimetric
inequality states that the measurable subset of $\setR^n$ with a
prescribed volume and minimal area is the ball of this volume. In this
paper, we
need Lévy's isoperimetric inequality, see
e.g.~\cite[Theorem~2.1]{figiel1977dimension}, which can be seen as an analogous
result for spheres:
it states that a measurable
subset of the sphere of prescribed area and minimal boundary
is a spherical cap.

% The notation $D(C)$ introduced above for the area of the ``dockable'' surface of a simplicial spherical cone $S_v$ can be extended to any (not necessarily convex) spherical cone $C$.
% We recall that the dockable surface is the area of the surface of $C$ which is in the interior of $B_n$ (i.e., the area of the lateral surface of $C$). We call \textit{base} of $C$ its intersection with the unit sphere. Furthermore, we introduce the following additional notation (see Figure~\ref{fig:cone-regions}):
% \begin{itemize}
% \item $B(C)$ for the area of the base of $C$.
% \item $L(C)$ for the ``length'' of the relative boundary of the base of $C$. (We use the term ``length'' to denote the measure of an $(n-2)$-dimensional volume).
% \end{itemize}

A spherical cone $S$ is a \textit{cone of revolution} if there exist a
vector $v$ and an angle $0<\theta\leq\pi/2$ such that $S$ is the set of
vectors in the unit ball that form an angle of at most $\theta$ with $v$:
\[S=\left\{x\in B_n:\frac{v^Tx}{\|v\| \|x\|}\geq\cos\theta\right\}.\]
Note that a spherical cone is a cone of revolution if and only if its base is a spherical cap. We also observe that two spherical cones of revolution, defined by two different vectors but by the same angle, are always congruent. Therefore, in the following we will only specify the angle of a cone of revolution.

\begin{lemma}
\label{lem:isoperim-cones}
The spherical cone of given volume with minimum lateral surface is a
cone of revolution.
\end{lemma}

\begin{proof}
  By the first equation of \eqref{eq:cones-formulas}, every spherical
  cone of volume $V$ intersects the unit sphere in a surface of area
  $nV$.  Furthermore, by the second equation of
  \eqref{eq:cones-formulas}, the length of the boundary of this
  surface is proportional to the area of the lateral surface of the
  cone. Then the problem of finding the spherical cone of volume $V$
  with the minimum lateral surface can be rephrased as follows: Find a
  surface of area $nV$ on the unit sphere having the boundary of
  minimum length. By Lévy's isoperimetric inequality for spheres,
  the optimal shape for such a surface is a spherical cap. As observed
  above, this corresponds to a cone of revolution.
\end{proof}

% \begin{lemma}
% \label{lem:half-ball}
% Let $C$ be half of the unit ball. Then $$\frac{L(C)}{B(C)}\geq\sqrt{\frac2\pi}\cdot\frac{n-1}{\sqrt{n}}.$$
% \end{lemma}

%\begin{proof}

%\end{proof}

\begin{lemma}
\label{lemma:LowerBoundOnSphericalCone}
Let $S$ be a spherical cone of revolution of angle $0<\theta\leq\pi/2$. Then $$\frac{D(S)}{\vol(S)}\geq\sqrt{\frac{2n}{\pi}}.$$
\end{lemma}

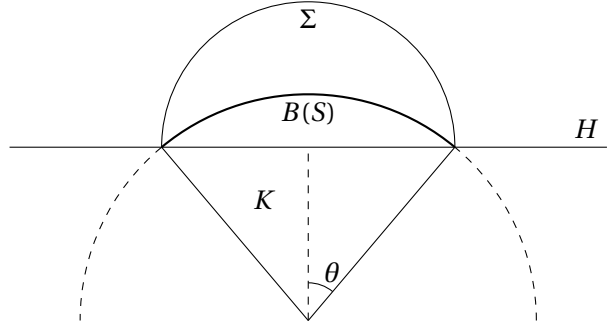
\begin{figure}
\centering
  \begin{tikzpicture}
    \draw[dashed] (0:3cm) arc (0:50:3cm);
    \draw[dashed] (130:3cm) arc (130:180:3cm);
    \draw (0,0) -- (50:3cm);
    \draw (130:3cm) -- (0,0);
    \draw[thick] (50:3cm) arc (50:130:3cm);
    \draw (0,2.75) node {$B(S)$};
    \draw (110:1.7cm) node {$K$};

    \draw (50:3cm) -- +(2,0) node[above left] {$H$};
    \draw (50:3cm) -- (130:3cm) -- +(-2,0);

    \draw[dashed] (0,0) -- (0,2.2981);
    \draw (50:0.5cm) node[above] {$\theta$} arc (50:90:0.5cm);

    \draw (50:3cm) arc (0:180:1.9284cm);
    \draw (0,4) node {$\mathbf{\Sigma}$};
  \end{tikzpicture}
  \caption{Proof of Lemma~\ref{lemma:LowerBoundOnSphericalCone}.}
  \label{fig:spherical-cone}
\end{figure}

\begin{proof}
  Using~\eqref{eq:cones-formulas}, we have to show that
  \begin{equation}
  \label{eq:9}
  \frac{L(S)}{B(S)}\geq\sqrt{\frac{2}{\pi}}\frac{n-1}{\sqrt{n}}.
\end{equation}

This is done in two steps. We first prove that this ratio is minimal
for $S$ being the half-ball, i.e., $\theta = \pi/2$. Then we show that
$\frac{L(S)}{B(S)}\geq\sqrt{\frac{2}{\pi}}\frac{n-1}{\sqrt{n}}$ holds for
the half-ball.

Let $H$ be the hyperplane containing the boundary of the base of $S$.
Then $H$ divides $S$ into two parts: a truncated cone $K$ and the
convex hull of a spherical cap. The radius $r$ of the base of $K$ is
bounded by one.

Consider now the half-ball that contains $B(S)$ and that has $H \cap
B_n$ as its flat-surface, see Figure~\ref{fig:spherical-cone}, and let
$\Sigma$ denote the area of the corresponding half-sphere.  One has
$B(S)\leq\Sigma$ and thus
\begin{displaymath}
  \frac{L(S)}{B(S)} \geq \frac{L(S)}{\Sigma}.
\end{displaymath}

Now $\Sigma$ and $L(S)$ are the surface of an $(n-1)$-dimensional
half-sphere of radius $r$ and the length of its boundary respectively.
If we scale this half-sphere by a factor of $1/r$, we obtain the unit
half-ball and its length respectively.  Since scaling by a factor of
$1/r$ increases areas by a factor of $1/r^{n-1}$ and lengths by a
factor of $1/r^{n-2}$, we have that $\frac{L(S)}\Sigma$ is at least the
length of the unit-half-ball divided by the area of the base of the
half-ball.

Suppose now that $S$ is the half-unit-ball.  We show that the inequality
$L(S)/B(S)\geq$ $\sqrt{\frac{2}{\pi}}\frac{n-1}{\sqrt{n}}$ holds.
The base of $S$ is a half unit sphere and $L(S)$ is the length of the
boundary of a unit ball of dimension $n-1$. Thus
$$B(S)=\frac{n}{2}\frac{\pi^{n/2}}{\Gamma\left(\frac n2+1\right)}, \quad L(S)=\frac{(n-1)\pi^{(n-1)/2}}{\Gamma\left(\frac {n-1}2+1\right)},$$
where $\Gamma$ is the well-known Gamma function.  Using the fact that
$\Gamma(x+1/2)/\Gamma(x)\geq\sqrt{x-\frac 14}$ for all $x>\frac 14$ (see, e.g.,
\cite{merkle1996logarithmic}), one easily verifies that
$$\Gamma\left(\frac n2+1\right)\geq\sqrt{\frac n2}\cdot\Gamma\left(\frac{n-1}2+1\right).$$
It follows that
$$\frac{L(S)}{B(S)}=\frac2{\sqrt\pi}\frac{n-1}{n}\frac{\Gamma\left(\frac
    n2+1\right)}{\Gamma\left(\frac
    {n-1}2+1\right)}\geq\sqrt{\frac2\pi}\cdot\frac{n-1}{\sqrt{n}}.$$
\end{proof}

Finally we are now ready to consider the case of an arbitrary
spherical cone.
\begin{lemma}
\label{lemma:LowerBoundOnDockingSurface}
Let $S$ be a (not necessarily convex) spherical cone with $\vol(S)\leq
\frac12 \vol(B_n)$. Then
$$\frac {D(S)}{\vol(S)}\geq \sqrt{\frac{2n}{\pi}}.$$
\end{lemma}

\begin{proof}
  Let $S^*$ be a spherical cone of revolution with the same volume as
  $S$. By Lemma~\ref{lem:isoperim-cones}, $D(S)\geq D(S^*)$. Now, using
  Lemma~\ref{lemma:LowerBoundOnSphericalCone} one has
  $$\frac{D(S)}{\vol(S)}\geq\frac{D(S^*)}{\vol(S^*)}\geq \sqrt{\frac{2n}{\pi}}.$$
\end{proof}

This was the final step in the proof of
Lemma~\ref{lemma:VolumeOfNeighborhood} and thus we have also proved
Theorem~\ref{thr:1},   our main
result on polytopes. The next section is devoted to unbounded
polyhedra.

\section{The case of an unbounded polyhedron}
\label{sec:unbounded}

If the polyhedron $P$ is unbounded, then the union of the normal cones of all vertices of $P$ forms a proper subset $K'$ of $\mathbb R^n$: namely, $K'$ is the set of objective functions $c$ for which the linear program $\max\{c^Tx:x\in P\}$ has finite optimum.
Similarly, the set $K'\cap B_n$ is a proper subset of $B_n$.
Then, given the union of the spherical cones that have already been discovered by the breadth-first-search (we denote this set by $S$),
we should redefine the dockable surface of $S$ as that part of the lateral surface of $S$ that is shared by some neighboring cones.
In other words, we should exclude the part lying on the boundary of $K'\cap B_n$.
However, this implies that Lemma~\ref{lemma:LowerBoundOnDockingSurface} cannot be immediately applied.

To overcome this difficulty, we make use of the Lovász-Simonovits
inequality, which we now recall. Below we use notation $d(X,Y)$ to
indicate the Euclidean distance between two subsets
$X,Y\subseteq\setR^n$, i.e., $d(X,Y)=\inf\{\norm{x-y}:x\in X,y\in
Y\}$. Also, $[x,y]$ denotes the segment connecting two points
$x,y\in\setR^n$ (see Figure~\ref{fig:lovasz-simonovits}).

\begin{theorem} \cite{lovasz1993random}
\label{theorem:LovaszSimonovitsInequality}
Let $K \subseteq \mathbb{R}^n$ be a convex compact set, $0<\epsilon<1$ and $(K_1,K_2,K_3)$ be a partition of $K$ into three measurable sets such that
\begin{equation}\label{eq:LovaszSimonovitsCondition}
\forall x,y \in K, \quad d([x,y] \cap K_1, [x,y] \cap K_2) \geq \epsilon\cdot \norm{x-y}.
\end{equation}
Then
\[\vol(K_3) \geq \frac{2\epsilon}{1-\epsilon} \min\left(\vol(K_1),\vol(K_2)\right).\]
\end{theorem}

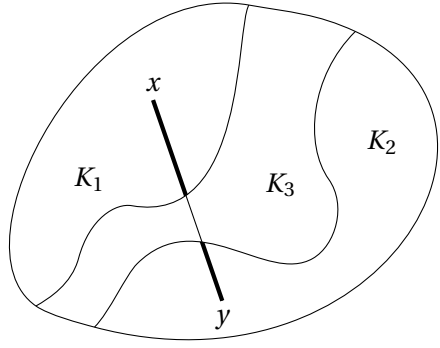
\begin{figure}
\centering
    \begin{tikzpicture}[scale=0.07]
    \coordinate (x) at (32,58);
    \coordinate (y) at (45,20);
    \coordinate (k1) at (20,43);
    \coordinate (k2) at (75,50);
    \coordinate (k3) at (56,42);
    \coordinate (f) at (10,19);
    \coordinate (g) at (18,28);
    \coordinate (h) at (28,38);
    \coordinate (i) at (50,76);
    \coordinate (j) at (70,71);
    \coordinate (k) at (65,43);
    \coordinate (l) at (62,28);
    \coordinate (m) at (29,27);
    \coordinate (n) at (21,15);
    \draw (x) -- (y);
    \draw (f) .. controls +(30:3) and +(255:5) .. (g) .. controls +(75:8) and +(170:3) .. (h) .. controls +(350:22) and +(240:5) .. (i);
    \draw (j) .. controls +(225:12) and +(125:8) .. (k) .. controls +(305:6) and +(30:5) .. (l) .. controls +(210:10) and +(45:16) .. (m) .. controls +(225:3) and +(45:4) .. (n);
    \draw (j) .. controls +(335:40) and +(345:55) .. (n) .. controls +(165:3) and +(325:3) .. (f) .. controls +(145:20) and +(170:30) .. (i) .. controls +(350:8) and +(155:8) .. (j);
    \begin{scope}
      \clip (f) .. controls +(30:3) and +(255:5) .. (g) .. controls +(75:8) and +(170:3) .. (h) .. controls +(350:22) and +(240:5) .. (i) .. controls +(170:30) and +(145:20) .. (f);
      \draw[ultra thick] (x) -- (y);
    \end{scope}
    \begin{scope}
      \clip (j) .. controls +(225:12) and +(125:8) .. (k) .. controls +(305:6) and +(30:5) .. (l) .. controls +(210:10) and +(45:16) .. (m) .. controls +(225:3) and +(45:4) .. (n) .. controls +(345:55) and +(335:40) .. (j);
      \draw[ultra thick] (x) -- (y);
    \end{scope}
    \draw (x) node[above] {$x$};
    \draw (y) node[below] {$y$};
    \draw (k1) node {$K_1$};
    \draw (k2) node {$K_2$};
    \draw (k3) node {$K_3$};
  \end{tikzpicture}
  \caption{Illustration of  the Lovász-Simonovits inequality.}
  \label{fig:lovasz-simonovits}
\end{figure}

We now illustrate how the above result can be used in our context.
Let $K = K'\cap B_n$ and observe that $K$ is a convex and compact set.
Let $S\subseteq K$ be the union of the spherical cones that have already been discovered by the breadth-first-search.
We define the dockable surface of $S$ as that part of the lateral surface of $S$ that is disjoint from the boundary of $K$.
We denote by $D'(S)$ the area of the dockable surface of $S$. We can prove the following analogue of Lemma~\ref{lemma:LowerBoundOnDockingSurface}:

\begin{lemma}
If $\vol(S) \leq \frac12 \vol(K)$, then
$D'(S)\geq \vol(S)$.
\end{lemma}

\begin{proof}
Let $F$ denote the dockable surface of $S$ (thus $D'(S)$ is the area of $F$).
For every $\epsilon>0$ we define
\begin{align*}
K_{3,\epsilon}&=(F+\epsilon B_n)\cap K,\\
K_{1,\epsilon}&=S\setminus K_{3,\epsilon},\\
K_{2,\epsilon}&=K\setminus(K_{1,\epsilon}\cup K_{3,\epsilon}),
\end{align*}
where $X+Y$ denotes the Minkowski sum of two subsets $X,Y\in\setR^n$, i.e., $X+Y=\{x+y:x\in X,y\in Y\}$.
Clearly $(K_{1,\epsilon},K_{2,\epsilon},K_{3,\epsilon})$ is a partition of $K$ into three measurable sets. Furthermore, condition~\eqref{eq:LovaszSimonovitsCondition} is satisfied. Thus Theorem~\ref{theorem:LovaszSimonovitsInequality} implies that
$$\frac{\vol(K_{3,\epsilon})}{2\epsilon}\geq\frac1{1-\epsilon}\min\left(\vol(K_{1,\epsilon}),\vol(K_{2,\epsilon})\right).$$
We observe that
\begin{equation*}
\begin{split}
\vol(K_{2,\epsilon}) &\geq \vol(K\setminus S)-\vol(K_{3,\epsilon})\\
& \geq \vol(S)-\vol(K_{3,\epsilon})\\
&\geq \vol(K_{1,\epsilon})-\vol(K_{3,\epsilon}).
\end{split}
\end{equation*}
Combining those two inequalities, we find
\begin{equation}\label{eq:unbounded-proof}
\frac{\vol(F+\epsilon B_n)}{2\epsilon}\geq\frac{\vol(K_{3,\epsilon})}{2\epsilon}\geq\frac1{1-\epsilon}(\vol(K_{1,\epsilon})-\vol(K_{3,\epsilon})).
\end{equation}
%Since $F$ is a bounded and rectifiable surface,
By a well-known result in geometry (see, e.g., \cite{Federer},) as $\epsilon$ tends to 0 the left-hand side of \eqref{eq:unbounded-proof} tends to the area of $F$,
which is precisely the dockable surface $D'(S)$.
Moreover, as $\epsilon$ tends to $0$, $\vol(K_{3,\epsilon})$ tends to $0$ and $\vol(K_{1,\epsilon})$ tends to $\vol(S)$.
We conclude that $D'(S) \geq \vol(S)$.
\end{proof}

Following the same approach as that used for the case of a polytope, one can show the following result for polyhedra.
\begin{theorem}
Let $P=\{x \in \setR^n \colon Ax\leq b\}$ be a polyhedron, where all sub-determinants of $A\in\Z^{m× n}$ are bounded by $\Delta$ in absolute value.
Then the diameter of $P$ is bounded by $O\left(\Delta^2 n^4\log n\Delta\right)$.
In particular, if $A$ is totally unimodular, then the diameter of $P$ is bounded by $O(n^4\log n)$.
\end{theorem}

\section{Remarks}

\subsection{Which sub-determinants enter the bound?}

For simplicity, we have assumed that a bound $\Delta$ was given for
the absolute value of all sub-determinants of $A$.  However, our proof
only uses the fact the the sub-determinants of size $1$ (i.e., the
entries of the matrix) and $n-1$ are bounded.  Calling $\Delta_1$
(resp. $\Delta_{n-1}$) the bound on the absolute value of the entries
of $A$ (resp. on the sub-determinants of $A$ of size $n-1$), one
easily verifies that all the results discussed above remain
essentially unchanged, except that the statement of
Lemma~\ref{lemma:UpperBoundOnDockingSurface} becomes
$$\frac{D(S_v)}{\vol(S_v)}\leq\Delta_1\Delta_{n-1}n^3$$
and the lower bound on $\vol(I_0)$ becomes
$$\vol(I_0)\geq \frac1{n!n^{n/2}\Delta_1^n}.$$
This implies the following strengthened result:

\begin{theorem}
  Let $P=\{x \in \setR^n \colon Ax\leq b\}$ be a polyhedron, where the
  entries of $A$ (respectively the sub-de­ter­mi­nants of $A$ of size
  $n-1$) are bounded in absolute value by $\Delta_1$ (respectively
  $\Delta_{n-1}$).  Then the diameter of $P$ is bounded by
  $O\left(\Delta_1\Delta_{n-1} n^4\log n\Delta_1\right)$. Moreover, if
  $P$ is a polytope, its diameter is bounded by
  $O\left(\Delta_1\Delta_{n-1} n^{3.5}\log n\Delta_1\right)$.
\end{theorem}

\subsection{A more general geometric setting}
\label{sec:geometric-setting}

Since our result was first announced in~\cite{BDEHN12}, Brunsch and
R\"oglin~\cite{brunsch2013finding} have found an algorithm to compute
a short path between two given vertices of a non-degenerate polyhedron
$P = \{ x \in \R^n \colon Ax \leq b\}$ that runs in expected
polynomial time in $n,m$ and $1/\delta$, where $\delta$ is a lower
bound on the sine of the angle of a row of $A$ to the subspace of
$n-1$ other rows of $A$.  The expected length of the path is $O(m n^2
/ \delta^2)$. If $A \in \Z^{m\times n}$, then $\delta \geq 1/(\Delta_1
\Delta_{n-1} n)$, where $\Delta_1$ and $\Delta_{n-1}$ are, as before,
bounds on the absolute values of $1\times1$ and $(n-1)\times(n-1)$
sub-determinants.

Our proof technique applies in this setting as well. We have volume
expansion since the normal cones cannot be too flat. The parameter
$\delta$ is a measure for this flatness. In this setting,
Lemma~\ref{lemma:VolumeOfNeighborhood} reads as follows.

\begin{lemma}
\label{lem:1}
Let $P=\{x \in \setR^n \colon Ax\leq b\}$ be a polytope and let
$I\subseteq V$ be a set of vertices with $\vol(I) \leq (1/2) \cdot
\vol(B_n)$.  Then the volume of the neighborhood of $I$ is at least
$$\vol({\mathcal{N}(I)})\geq \sqrt{\frac{2}{\pi}}\left({\delta}/{n^{1.5}}\right)\cdot \vol(I).$$
\end{lemma}

\noindent
The proof is along the lines of the proof of
Lemma~\ref{lemma:VolumeOfNeighborhood} and by adapting
Lemma~\ref{lemma:UpperBoundOnDockingSurface}. Here one has now
\begin{displaymath}
  D(S_v) / \vol(S_v) \leq n^2 / \delta.
\end{displaymath}
Theorem~\ref{thr:1} is in the geometric setting now becomes the following.

\begin{theorem}
  Let $P=\{x \in \setR^n \colon Ax\leq b\}$ be a polytope where the
  sine of the angle of any row of $A$ to the subspace generated by
  $n-1$ other rows of $A$ is at least $\delta$.  The diameter of $P$
  is bounded by $O\left(n^{2.5}/\delta \cdot \ln (n/\delta) \right)$.
%In particular, if $A$ is totally unimodular, then the diameter of $P$ is bounded by $O(n^{3.5}\log n)$.
\end{theorem}

\noindent
Again, the proof is along the lines of the proof of
Theorem~\ref{thr:1} where the volume of $S_v$ is now lower bounded by
$\delta^{n-1} / n!$. In fact, the diameter bound
$O\left(n^{2.5}/\delta \cdot \ln (n/\delta) \right)$ holds already for
non-degenerate polytopes where each $S_v$ contains a ball of radius
$\delta$. For polyhedra, we obtain a bound of
\begin{displaymath}
  O\left(n^{3}/\delta  \cdot \ln (n/\delta) \right)
\end{displaymath}
on the diameter.

\subsubsection*{Acknowledgments} 
This work was carried out while all authors were at EPFL (École
Polytechnique Fédérale de Lausanne), Switzerland.  The authors
acknowledge support from the DFG Focus Program 1307 within the project
``Algorithm Engineering for Real-time Scheduling and Routing'' and
from the Swiss National Science Foundation within the project
``Set-par­ti­tion­ing integer programs and integrality gaps''.

{
\small
%\bibliographystyle{alphaabbr}
%\bibliography{books,mybib,papers,my_publications}

\begin{thebibliography}{DLKOS09}

\bibitem[Bal84]{MR769400}
M.~L. Balinski.
\newblock The {H}irsch conjecture for dual transportation polyhedra.
\newblock {\em Math. Oper. Res.}, 9(4):629--633, 1984.

\bibitem[Bar74]{MR0355826}
D.~Barnette.
\newblock An upper bound for the diameter of a polytope.
\newblock {\em Discrete Math.}, 10:9--13, 1974.

\bibitem[BDSE{\etalchar{+}}12]{BDEHN12}
N.~Bonifas, M.~Di~Summa, F.~Eisenbrand, N.~H\"{a}hnle, and M.~Niemeier.
\newblock On sub-determinants and the diameter of polyhedra.
\newblock In {\em Proceedings of the 28th annual ACM symposium on Computational
  geometry}, SoCG '12, pages 357--362, 2012.

\bibitem[BR13]{brunsch2013finding}
T.~Brunsch and H.~R{\"o}glin.
\newblock Finding short paths on polytopes by the shadow vertex algorithm.
\newblock In {\em Automata, Languages, and Programming}, pages 279--290.
  Springer, 2013.

\bibitem[BvdHS06]{MR2223631}
G.~Brightwell, J.~van~den Heuvel, and L.~Stougie.
\newblock A linear bound on the diameter of the transportation polytope.
\newblock {\em Combinatorica}, 26(2):133--139, 2006.

\bibitem[DF94]{MR1274170}
M.~Dyer and A.~Frieze.
\newblock Random walks, totally unimodular matrices, and a randomised dual
  simplex algorithm.
\newblock {\em Mathematical Programming}, 64(1, Ser. A):1--16, 1994.

\bibitem[DLKOS09]{MR2568801}
J.~A. De~Loera, E.~D. Kim, S.~Onn, and F.~Santos.
\newblock Graphs of transportation polytopes.
\newblock {\em J. Combin. Theory Ser. A}, 116(8):1306--1325, 2009.

\bibitem[EHRR10]{EHRR10}
F.~Eisenbrand, N.~Hähnle, A.~Razborov, and T.~Rothvo\ss.
\newblock Diameter of polyhedra: {L}imits of abstraction.
\newblock {\em Mathematics of Operations Research}, 35(4):786--794, 2010.

\bibitem[Fed69]{Federer}
H.~Federer.
\newblock {\em Geometric Measure Theory}.
\newblock Springer, 1969.

\bibitem[FLM77]{figiel1977dimension}
T.~Figiel, J.~Lindenstrauss, and V.~Milman.
\newblock The dimension of almost spherical sections of convex bodies.
\newblock {\em Acta Mathematica}, 139(1):53--94, 1977.

\bibitem[Kai04]{MR2077562}
V.~Kaibel.
\newblock On the expansion of graphs of 0/1-polytopes.
\newblock In {\em The sharpest cut}, MPS/SIAM Ser. Optim., pages 199--216.
  SIAM, Philadelphia, PA, 2004.

\bibitem[Kal91]{MR1116365}
G.~Kalai.
\newblock The diameter of graphs of convex polytopes and {$f$}-vector theory.
\newblock In {\em Applied geometry and discrete mathematics}, volume~4 of {\em
  DIMACS Ser. Discrete Math. Theoret. Comput. Sci.}, pages 387--411. Amer.
  Math. Soc., Providence, RI, 1991.

\bibitem[KK92]{MR1130448}
G.~Kalai and D.~J. Kleitman.
\newblock A quasi-polynomial bound for the diameter of graphs of polyhedra.
\newblock {\em Bull. Amer. Math. Soc. (N.S.)}, 26(2):315--316, 1992.

\bibitem[KS10]{MR2681516}
E.~D. Kim and F.~Santos.
\newblock An update on the {H}irsch conjecture.
\newblock {\em Jahresber. Dtsch. Math.-Ver.}, 112(2):73--98, 2010.

\bibitem[KW67]{KleeWalkup}
V.~Klee and D.~W. Walkup.
\newblock The $d$-step conjecture for polyhedra of dimension $d<6$.
\newblock {\em Acta Math. 133}, pages 53--78, 1967.

\bibitem[Lar70]{MR0254735}
D.~G. Larman.
\newblock Paths of polytopes.
\newblock {\em Proc. London Math. Soc. (3)}, 20:161--178, 1970.

\bibitem[LS93]{lovasz1993random}
L.~Lov{\'a}sz and M.~Simonovits.
\newblock Random walks in a convex body and an improved volume algorithm.
\newblock {\em Random structures \& algorithms}, 4(4):359--412, 1993.

\bibitem[Mer96]{merkle1996logarithmic}
M.~Merkle.
\newblock Logarithmic convexity and inequalities for the gamma function.
\newblock {\em Journal of mathematical analysis and applications},
  203(2):369--380, 1996.

\bibitem[Nad89]{Naddef89}
D.~Naddef.
\newblock The {H}irsch conjecture is true for (0,1)-polytopes.
\newblock {\em Mathematical Programming}, 45:109--110, 1989.

\bibitem[Orl97]{MR1486304}
J.~B. Orlin.
\newblock A polynomial time primal network simplex algorithm for minimum cost
  flows.
\newblock {\em Mathematical Programming}, 78(2, Ser. B):109--129, 1997.
\newblock Network optimization: algorithms and applications (San Miniato,
  1993).

\bibitem[San12]{Santos10}
F.~Santos.
\newblock A counterexample to the {H}irsch conjecture.
\newblock {\em Ann. of Math. (2)}, 176(1):383--412, 2012.

\bibitem[Ver09]{MR2529774}
R.~Vershynin.
\newblock Beyond {H}irsch conjecture: walks on random polytopes and smoothed
  complexity of the simplex method.
\newblock {\em SIAM J. Comput.}, 39(2):646--678, 2009.

\end{thebibliography}

\newcommand{\etalchar}[1]{$^{#1}$}

}

\end{document}